\documentclass[12pt,leqno]{amsart}

\usepackage[a4paper,margin=25mm]{geometry}
\usepackage{amsmath,amssymb}

\newtheorem{theorem}{Theorem}
\newtheorem{corollary}[theorem]{Corollary}
\newtheorem*{conjecture}{Conjecture}

\title{Kobayashi visibility of strictly convex domains}
\author{Nikolai Nikolov and Amar Deep Sarkar}

\address{NN: Institute of Mathematics and Informatics, Bulgarian Academy
of Sciences, Acad. G. Bonchev 8, 1113 Sofia, Bulgaria
\vspace{1mm}
\newline Faculty of Information Sciences,
State University of Library Studies and Information Technologies,
Shipchenski prohod 69A, 1574 Sofia, Bulgaria}
\email{nik@math.bas.bg}

\address{ADS: Department of Mathematics,
Indian Institute of Technology Bhubaneswar, Argul 752050, India}
\email{amar@iitbbs.ac.in}

\thanks{The first-named author was partially supported by the Bulgarian National Science
Fund, Ministry of Education and Science of Bulgaria under contract KP-06-N82/6.}


\begin{document}

\begin{abstract} We show that any bounded strictly convex domain in $\Bbb C^N$,
endowed with the Kobayashi distance, is visible.
\end{abstract}

\maketitle

Let $D\subset\Bbb R^n$ be a bounded domain and let $(D,d)$ be a geodesic metric space.
Following \cite{BNT}, we say that $(D,d)$ is \emph{visible} if, for any two distinct
points $p,q\in\partial D$, there exist neighborhoods $U$ of $p$ and $V$ of $q$ and a
compact set $K\subset D$ such that every geodesic joining a point of $U\cap D$ to a
point of $V\cap D$ intersects $K$.

Recall that a domain $D\subsetneq\Bbb R^n$ is \emph{strictly convex} if it is
convex and $\partial D$ contains no nontrivial line segment.

\begin{theorem}\label{1}
Let $D\subset\Bbb R^n$ be a bounded strictly convex domain. Assume that $(D,d)$ is a proper
geodesic metric space, $d$ induces the Euclidean topology on $D$, and the $d$-balls are
convex in the Euclidean sense. Then $(D,d)$ is visible.
\end{theorem}

\begin{proof} Fix $o\in D$ and set
$$(x|y)_o=\frac12\bigl(d(x,o)+d(y,o)-d(x,y)\bigr),\qquad x,y\in D.$$
Although \cite[Proposition~2.5]{BNT} is stated for $k_D$ ($D\subset\Bbb C^N)$, its proof uses
only compactness of closed bounded balls and the agreement of the metric and Euclidean
topologies. Therefore $(D,d)$ is visible if and only if for any two distinct points
$p,q\in\partial D,$
\begin{equation}\label{gromov}
\limsup_{(x,y)\to(p,q)}(x|y)_o<\infty.
\end{equation}

Suppose that $(D,d)$ is not visible. Then there exist two distinct points
$p,q\in\partial D$ and sequences $x_j\to p$, $y_j\to q$ such that
$(x_j|y_j)_o\to+\infty.$ Let $\sigma_j$ be a $d$-geodesic from $x_j$ to $y_j$.
Because $0\le d(o,x_j)-(x_j|y_j)_o\le d(x_j,y_j)$, we may choose $z_j\in\sigma_j$
so that $d(x_j,z_j)=d(o,x_j)-(x_j|y_j)_o.$ By boundedness of $D$, after passing to a subsequence,
$z_j\to r\in\overline D$. Since $d(o,z_j)\ge (x_j|y_j)_o\to+\infty$, the agreement of the two
topologies implies $r\in\partial D$. The identity $d(y_j,z_j)-d(y_j,o)=-(x_j|y_j)_o$ shows that,
interchanging $x_j$ and $y_j$ if necessary, we may assume that $p\neq r$. Since
$d(x_j,z_j)-d(x_j,o)=-(x_j|y_j)_o\to-\infty$ and the $d$-balls are convex, \cite[Proposition~4.3]{HW}
yields $[r,p]\subset\partial D$, a contradiction to the strict convexity of $D$.
\end{proof}

\begin{corollary}\label{2}
Let $D\subset\Bbb R^n$ be a bounded convex domain. Then $D$, endowed with the
Hilbert metric, is visible if and only if it is strictly convex.
\end{corollary}

\begin{proof}
Sufficiency follows from Theorem~\ref{1}, since the Hilbert metric satisfies its
metric assumptions; see, e.g., \cite[Section~2]{HW}. For the necessity, it is enough to
observe that if $[a,b]\subset\partial D$, $a\neq b$, and $o\in D$, then
$[to+(1-t)a,\,to+(1-t)b]$, $0<t<1$, are geodesic segments in $D$.
\end{proof}

Note that, in view of \eqref{gromov}, sufficiency is equivalent to the strictly
convex case of \cite[Theorem~5.2]{KN}.

We now turn to the Kobayashi distance. Recall that a domain $D\subsetneq\Bbb C^N$ is called strictly 
(resp.~$\Bbb C$-strictly)
convexifiable at $p\in\partial D$ if there exist a neighborhood $U$ of $p$ and a biholomorphism
$\Phi:U\to\Phi(U)$ such that $G=\Phi(D\cap U)$ is convex, and every real (resp.~complex)
line disjoint from $G$ meets $\partial G$ in at most one point.

The main consequence of Theorem~\ref{1} is the following.

\begin{corollary}\label{3}
Let $D\subset\Bbb C^N$ be a bounded domain which is strictly
convexifiable at every boundary point. Then $(D,k_D)$ is visible.
\end{corollary}

Strictness is essential, as the unit polydisc shows.

\begin{proof}
Fix $p\in\partial D$ and choose $U,\Phi$ as in the definition. Let $B$ be a ball
centered at $\Phi(p)$ with $\overline B\subset\Phi(U)$. Then $G'=\Phi(D\cap U)\cap B$ is bounded
and strictly convex. The metric assumptions of Theorem~\ref{1} hold for $(G',k_{G'})$;
see, e.g., \cite{A}. Hence $D\cap\Phi^{-1}(B)$, being biholomorphic to $G'$ via a map
extending homeomorphically to their closures, is complete hyperbolic
and, by Theorem~\ref{1}, visible. Since $D$ is bounded and locally complete hyperbolic,
it is complete hyperbolic (see, e.g., \cite[Theorem~14.5.5]{JP}). The conclusion then
follows from \cite[Theorem~4.1]{NOT}.
\end{proof}

Thus, Corollary~\ref{3} supports the following (see also \cite[Conjecture~3.5]{BO}).

\begin{conjecture}
Let $D\subset\Bbb C^N$ be a bounded domain which is $\Bbb C$-strictly
convexifiable at every boundary point. Then $(D,k_D)$ is visible.
\end{conjecture}

The conjecture follows from
\cite[Corollary~4.2 and the proof of Theorem~1.3]{BNT}
when $k_D$ has $\frac12$-log-growth, as is the case, for instance,
when $\partial D$ is Dini-smooth; see also
\cite[Theorem~4.1(2)]{Z} for bounded $\Bbb C$-strictly convex
domains with $C^{1,\alpha}$-smooth boundary.
On the other hand, Corollary~\ref{3} requires no boundary regularity.
\medskip

\noindent{\bf Remark.} Let $D\subset\Bbb C^N$ be a bounded strictly convex domain.
Any two distinct points of $D$ lie on a unique complex geodesic, up to precomposition
with an automorphism of the unit disc $\Bbb D$; see, e.g.,
\cite[Theorem~11.2.1 and Proposition~11.3.3]{JP} and \cite{V}.
Recall that a complex geodesic is a
holomorphic isometric embedding $\varphi:(\Bbb D,k_{\Bbb D})\rightarrow(D,k_D)$.

In fact, more is true: by \cite[Theorem~3]{V}, all $k_D$-balls
are strictly convex, which easily implies that any two distinct
points of $D$ are joined by a unique real $k_D$-geodesic.

In this direction, note that combining Corollary \ref{3} with \cite[Proposition~4 and Corollary~6]{N}
gives the following: every complex geodesic of $D$ extends continuously to $\overline{\Bbb D}$ and
maps $\partial\Bbb D$ into $\partial D$; for any two distinct points $p,q\in\overline D$, there
exists a complex geodesic $\varphi$ such that $p,q\in\varphi(\overline{\Bbb D})$.
\medskip

\noindent{\bf Acknowledgement.} We would like to thank the anonymous referee
for their constructive remarks.

\end{document}